\documentclass[12pt,a4paper]{amsart}

\usepackage{fullpage}
\usepackage{amsmath}
\usepackage{amsfonts}
\usepackage{amssymb}
\usepackage{amsthm}

\usepackage{latexsym}
\usepackage[mathscr]{eucal}
\usepackage{times}

\usepackage{thmtools, thm-restate}

\usepackage{tikz}
\usetikzlibrary{backgrounds,calc}
\usetikzlibrary{calc,arrows}

\usepackage{pgffor}

\newtheorem{thm}{Theorem}[section]

\newtheorem{cor}[thm]{Corollary}
\newtheorem{lem}[thm]{Lemma}
\newtheorem{prop}[thm]{Proposition}
\newtheorem{teo}[thm]{Theorem}

\theoremstyle{definition}
\newtheorem{defi}[thm]{Definition}

\newtheorem{rem}[thm]{Remark}

\DeclareMathOperator{\Disp}{Disp}

\DeclareMathOperator{\proj}{proj}
\DeclareMathOperator{\Iso}{Iso}
\DeclareMathOperator{\supp}{supp}
\DeclareMathOperator{\Geo}{Geo}
\DeclareMathOperator{\OptGeo}{OptGeo}
\DeclareMathOperator{\Tan}{Tan}
\DeclareMathOperator{\Lip}{Lip}
\DeclareMathOperator{\Ch}{Ch}

\DeclareMathOperator{\vol}{vol}
\DeclareMathOperator{\Img}{Img}
\DeclareMathOperator{\Id}{Id}
\DeclareMathOperator{\Fix}{Fix}

\begin{document}

 
\title{On the isometry group of $RCD^*(K,N)$-spaces.}

\author[L. Guijarro]{Luis Guijarro 
}
\address{Department of Mathematics, Universidad Aut\'{o}noma de Madrid, and
ICMAT CSIC-UAM-UCM-UC3M }
\email{luis.guijarro@uam.es}
\urladdr{http://www.uam.es/luis.guijarro}
\thanks{Both authors were supported by research grants MTM2014-57769-3-P and  MTM2017-85934-C3-2-P
(MINECO) and ICMAT Severo Ochoa project SEV-2015-0554 (MINECO). The second author was supported by a PhD Scholarship awarded by Conacyt.}
\author[J. Santos-Rodr\'iguez]{Jaime Santos-Rodr\'iguez
}
%
%
%

\address{Department of Mathematics, Universidad Aut\'onoma de Madrid, Spain}
\email{jaime.santosr@estudiante.uam.es}

\date{\today}




\begin{abstract}
We prove that the group of isometries of a metric measure space that satisfies  the  Riemannian curvature condition, $RCD^*(K,N),$ is a Lie group. We obtain an optimal upper bound on the dimension of this group, and classify the spaces where this maximal dimension is attained.
\end{abstract}
\maketitle

\section{Introduction.}

The notion of synthetic  lower Ricci curvature bounds was independently defined for metric measure spaces by  Sturm  (\cite{SturmI}, \cite{SturmII}) and by Lott-Villani (\cite{Vill}). This condition, called \emph{curvature-dimension condition}  $CD(K,N),$ is given in terms of the optimal transport between probability measures on a metric space and of the convexity of an entropy functional. It is known that a Riemannian manifold $(M,g)$ having Ricci curvature bounded below is equivalent to the metric measure space $(M,d_g, d\vol_g)$ being a $CD(K,N)$ space.  

We say that a metric measure space $(X,d,\mathfrak{m})$ satisfies the Riemannian curvature dimension condition $RCD(K,N)$ (also called \emph{infinitesimally Hilbertian})  if the associated Sobolev space $W^{1,2}$ is a Hilbert space \cite{Gigli2}. This condition rules out Finsler geometries and provides good  structural results such as an extension of the classical Splitting theorem of Cheeger-Gromoll done by Gigli \cite{Gigli}. In fact, this last result does not  extend to the larger class of  $CD(K,N)$ spaces as one can see by considering the case of $(\mathbb{R}^n,d_{\|\cdot\|}, \mathcal{L}^n),$ where $d_{\|\cdot\|}$ is a distance induced by a non-Euclidean norm and $\mathcal{L}^n$ is the Lebesgue measure. If the norm does not come from a interior product then the space is $CD(0,n)$, has plenty of lines, but does not split. For details  we refer the reader to  \cite{Vill}.
  
In \cite{BachSturm} Bacher-Sturm defined via some distortion coefficients a curvature dimension condition that has better local-to-global and tensorization properties. This condition is called \emph{reduced curvature dimension condition} and is denoted by $CD^*(K,N)$.  It is worth noting that  every $CD(K,N)$ space is also a $CD^*(K,N)$ space, and that a $CD^*(K,N)$ space is $CD(\frac{N-1}{N}K,N)$.  A space that is both $CD^*(K,N)$ and infinitesimally Hilbertian will be denoted as $RCD^*(K,N)$. 
Gigli-Mondino-Rajala \cite{GigliMondRaj} proved that almost any point in an 
 $RCD^*(K,N)$-space has a tangent cone (not necessarily unique) that is Euclidean. This was later improved by Mondino and Naber \cite{MonNab} to an stratification of the space into subsets of points with unique Euclidean tangent spaces, though the dimension of  such tangents might vary point-wise. 
%

Isometric actions on Riemannian manifolds have been a useful tool to investigate the interaction between the topology and the geometry of a Riemannian metric. 
A major result in this area  is the  theorem of Myers-Steenrod \cite{MyersSteen} stating that the isometry group of a Riemannian manifold is  a Lie group.  A good reference for this and similar results is \cite{Koba}.

The aim of this paper is to study the isometry group of $RCD^*(K,N)$-spaces. 
We start by obtaining an analogue of the Myers-Steenrod Theorem. 

\begin{restatable}{teorema}{liegroup}
\label{thm:lie group}
Let $(X,d_{X},\mathfrak{m})$ be an $RCD^{*}(K,N)-$space with $K\in \mathbb{R}$,  $N\in [1,\infty)$. Then 
 the isometry group of $X$ is a Lie group.
\end{restatable}

This theorem was proved in \cite{CheegCol2} for non-collapsed Gromov-Hausdorff limits of Riemannian manifolds with lower Ricci curvature bounds. 


 
 We  also bound the dimension of the isometry group in terms of $N$.

\begin{restatable}{teorema}{dimensionbound} 
\label{thm:dimension}
Let $(X,d_{X},\mathfrak{m})$ be an $RCD^*(K,N)$ space with $K,N\in \mathbb{R}$ and $N<\infty$. Then the Hausdorff dimension of  the isometry group of $X$ is bounded above by $\frac{{\lfloor N \rfloor} ({\lfloor N \rfloor} +1) }{2}$.
\end{restatable}


 Finally, we describe the spaces whose isometry group attains the maximal dimension described in the previous Theorem. 

\begin{restatable}{teorema}{maximaldimension}
\label{thm:rigidity}
Let $(X, d_{X}, \mathfrak{m})$  be an   $RCD^*(K,N)$--space 
. If the dimension of the isometry group of  $X$ is ${\lfloor N \rfloor} ({\lfloor N \rfloor} +1)/2,$ then $X$ is isometric to one of the following space forms:
\begin{enumerate}
\item An ${\lfloor N \rfloor}$--dimensional Euclidean space $\mathbb{R}^{{\lfloor N \rfloor} }$.
\item An ${\lfloor N \rfloor}$--dimensional sphere $\mathbb{S}^{{\lfloor N \rfloor} }$.
\item An ${\lfloor N \rfloor}$--dimensional projective space $\mathbb{R}P^{{\lfloor N \rfloor} }$.
\item An ${\lfloor N \rfloor}$--dimensional simply connected hyperbolic space $\mathbb{H}^{{\lfloor N \rfloor} }$.
\end{enumerate}
\end{restatable}


For Alexandrov spaces, Theorem \ref{thm:lie group} was proven in \cite{FukYam}, and  Theorems \ref{thm:dimension} and \ref{thm:rigidity} were proven in \cite{GalGui}. 

The paper is organized as follows: Section 2 gives preliminaries and general facts needed  in the rest of the paper; Section 3 gives some background on the isometry group of a general metric space, while in Section 4, we prove that the isometry group of an $RCD^*(K,N)$-space is a Lie group;  in Section 5, we bound the maximal dimension of the isometry group of an $RCD^*(K,N)$, and finally, Section 6 deals with the rigidity of a space when the dimension of its isometry group is maximal. 

After a first version of this work was completed and posted to the ArXiv, Gerardo Sosa (MPI, Leipzig) posted a preprint with a result similar to our Theorem \ref{thm:lie group}. His approach is different to ours, and applies to a broader class of metric spaces, although he does not studies Theorems \ref{thm:dimension} and \ref{thm:rigidity}. 

\section{Preliminaries.}
We present here some of the background necessary for the rest of the paper. 

\subsection{Dimension theory} A standard reference for much of this section is \cite{HuWall}. To simplify, we assume that $X$ is a metric space. 

Let $\mathcal{O}$ an open covering of  $X$; its order is defined as the smallest integer $n$ such that every point in $X$ is contained in at most $n$ members of the cover. 

\begin{defi}
The \emph{topological dimension} of $X$, $\dim_T X$ is defined as the  minimum value $n$ such that any open covering of $X$ has an open refinement with order $n+1$ or less.
\end{defi}

The topological dimension of $n$-dimensional manifolds agrees with $n$; thus, for a Lie group with a closed subgroup $H$, we have that 
\[
\dim_T G=\dim_T G/H + \dim_T H.
\]

A subspace $Y$ of $X$ satisfies $\dim_TY\leq \dim_TX$; thus, if $f:Z\to X$ is a topological embedding, we have  $\dim_TZ\leq \dim_TX$. 
Therefore, if $Z$ is compact, and $f:Z\to X$ is a continuous bijection onto its image, we get that $\dim_T Z\leq \dim_T X$. 

Since $X$ is a metric space, there is also the \emph{Hausdorff dimension} $\dim_{\mathcal{H}} X$; the relation between them is given by the inequality
$
\dim_TX\leq \dim_{\mathcal{H}} X
$
(see \cite[Page 107]{HuWall}).

\subsection{Metric measure spaces}
We will only consider metric measure spaces $(X,d,\nu)$ where $\nu$ is a Radon measure over Borel sets.
Two metric measure spaces $(X,d_{X},\mu)$ and $(Y,d_{Y},\nu)$ are \emph{isomorphic} if there exists an isometry $T: (\supp(\mu),d_{X}) \rightarrow (\supp (\nu), d_{Y})$ such that $T_{\#}\mu = \nu$. Observe that every metric measure space $(X,d_{X}, \mu)$ is isomorphic to $(\supp(\mu), d_{X},\mu)$, so from this point on we will assume that $\supp(\mu)=X$. We will denote by $\mathbb{P}_2(X)$ the space of probability measures on $X$ with finite second moments, that is 
\[
\int_X d^2(x_0,x)d\mu (x) <\infty,    
\] for some (and hence for all) $x_0 \in X. $

Let $\mu_0, \mu_1 \in \mathbb{P}_2(X)$. A measure $\pi \in \mathbb{P}(X\times X)$ is a \emph{coupling} between $\mu_0$ and $\mu_1$ if for the projections $\proj_i: X\times X \rightarrow X,$ $i=1,2$,  we have $\proj_{1\#}\pi = \mu_0$, $\proj_{2\#}\pi=\mu_1$. We define the  \emph{$L^2$-Wasserstein distance} between $\mu_0$ and $\mu_1$ as:
\begin{equation*}
\mathbb{W}_2^2(\mu_0,\mu_1) := 
\inf_\pi \left\lbrace\,\int_{X\times X}d(x,y)^2d\pi(x,y)\, \bigg| \, \pi \text{ is a coupling of } \mu_0,\mu_1 \,\right\rbrace.
\end{equation*} A measure $\pi$ that satisfies the infimum will be called an \emph{optimal coupling} between $\mu_0$ and $\mu_1$.


A curve $\gamma: [0,l] \rightarrow X$ is a \emph{geodesic} if its length coincides with the distance between its endpoints. We will call a metric space $(X,d)$ \emph{a geodesic space} if for any two points in $X$ there exists a geodesic that connects them. 
 
 If $(X,d)$ is a Polish geodesic space, then the $L^2$-Wasserstein space $(\mathbb{P}_2(X), \mathbb{W}_2)$ is also Polish and geodesic  (see for example \cite{SturmI}, \cite{Vill}). Let $\Geo(X)$ denote the set of constant speed geodesics  from $[0,1]$ to $(X,d)$ equipped with the $\sup$ norm.  

 We can lift any geodesic $(\mu_t)_{t\in [0,1]} \in \Geo(\mathbb{P}_2(X))$ to a  measure $\nu \in \mathbb{P}(\Geo(X)),$ in such a way that $(e_t)_{\#}\nu= \mu_t$, where $e_t: \Geo(X)\rightarrow X$ is the evaluation map at $t\in [0,1]$. We will then denote by $\OptGeo(\mu_0,\mu_1)$ the space of probability measures $\nu \in \mathbb{P}(\Geo(X))$ such that $(e_0,e_1)_{\#}\nu$ is an optimal coupling between $\mu_0$ and $\mu_1$.

 We will say that an optimal coupling $\nu \in \mathbb{P}(\Geo(X))$ is \emph{essentially non-branching} if it is concentrated on a set of non-branching geodesics.

\subsection{Pointed Gromov-Hausdorff convergence}
\begin{defi}
Let $(X,d_X, p_X),$ $(Y,d_Y, p_Y)$ be Polish  metric spaces. For $\epsilon >0 $ we will call a function  $f_\epsilon : B_{p_X}(1/\epsilon, X)\rightarrow B_{p_Y}(1/\epsilon, Y) $ a pointed $\epsilon$-Gromov-Hausdorff approximation if:
\begin{enumerate}
\item $f_\epsilon(p_X)= p_Y,$

\item for all $u,v \in B_{p_X}(1/\epsilon, X)$, $|\,d_X(u,v)-d_Y(f_\epsilon(u),f_\epsilon(v))\,|< \epsilon,$ and

\item for all $y \in B_{p_Y}(1/\epsilon, Y)$, there exists $x \in B_{p_X}(1/\epsilon, X)$ such that $d_Y(f_\epsilon(x),y) < \epsilon$.
\end{enumerate}
 Note that we do not assume that $f_\epsilon$ is continuous.
\end{defi}

The pointed Gromov Hausdorff distance between two pointed metric spaces $(X,d_X, p_X),$ $(Y,d_Y,p_Y),$ denoted by $d_{GH}((X,d_X, p_X),(Y,d_Y,p_Y))$, is by definition equal to the infimum of all $\epsilon$ such that there exist $\epsilon$-Gromov-Hausdorff approximations from  $(X,d_X, p_X)$ to  $(Y,d_Y,p_Y)$, and from  $(Y,d_Y,p_Y)$ to $(X,d_X, p_X)$.

\subsection{Equivariant Gromov-Hausdorff convergence}

Next, we introduce the notion of equivariant Gromov-Hausdorff convergence  as appears in \cite{FukYam2}. For $r >0 $ and $\Gamma$, a group acting by isometries on a pointed metric space $(X, d_X, p),$ we define $\Gamma (r) := \lbrace \gamma \in \Gamma \,|\, \gamma p \in B_p (r) \rbrace$.

\begin{defi}
Let $(X,d_X, p, \Gamma)$ and  $ (Y,d_Y,q,\Lambda)$ be two metric spaces and $\Gamma, \Lambda$ groups acting isometrically on $X$ and $Y$ respectively. An \emph{equivariant pointed  $\epsilon$--Gromov-Hausdorff approximation} is a triple $(f,\varphi,\psi)$ of maps 
\[
f: B_p(1/\epsilon, X)\rightarrow B_q(1/\epsilon,Y),
\quad 
 \varphi : \Gamma (1/\epsilon)\rightarrow \Lambda (1/\epsilon),
\quad 
\psi: \Lambda(1/\epsilon)\rightarrow \Gamma (1/\epsilon)
\]  
such that
\begin{enumerate}
\item $f(p)=q;$
\item the $\epsilon$--neighbourhood of $f\left(B_p(1/\epsilon,X)\right)$ contains $B_q(1/\epsilon,Y);$
\item if $x,y \in B_p(1/\epsilon,X),$ then 
\[
\left|\,d_X(x,y)-d_Y(f(x),f(y))\,\right|< \epsilon;
\]
\item if $\gamma \in \Gamma (1/\epsilon)$ and both $x, \gamma x \in B_p(1/\epsilon,X), $ then $$d_Y(f(\gamma x), \varphi (\gamma)f(x)) < \epsilon; $$
\item if $\lambda \in \Lambda (1/\epsilon)$ and both $x, \psi (\lambda)x \in B_p(1/\epsilon,X),$ then  $$d_Y(f(\psi (\lambda)x), \lambda f(x)) < \epsilon. $$
\end{enumerate} As usual, we do not assume that $f$ is continuous or that $\varphi, \psi$ are group homomorphisms.

\end{defi}

\begin{defi}
A sequence $(X_i,d_{X_i}, p, \Gamma_i)$ converges to  $(Y,d_Y,q,\Lambda)$ if there are equivariant pointed  $\epsilon_i$-Gromov-Hausdorff approximations with $\epsilon_i\to 0$.
\end{defi}

The next result links  the usual Gromov-Hausdorff convergence to its equivariant version. 
\begin{thm}[Proposition 3.6, \cite{FukYam2}]
\label{thm:fukaya-yamaguchi}
Let $(X_i,d_{X_i}, p, \Gamma_i)$ such that $(X_i,d_{X_i}, p)$ converges to $(Y,d_Y,q)$ in the pointed Gromov-Hausdorff distance. Then there exist some group of isometries of $Y$, $\Lambda$, such that by passing to a subsequence if necessary, we have that
\[
(X_{i_k},d_{X_{i_k}}, p, \Gamma_{i_k}) \to (Y,d_Y,q,\Lambda).
\]
\end{thm}

%

\subsection{Measured Gromov-Hausdorff convergence}

\begin{defi}
If  $(X_k,d_k, \mathfrak{m}_k, p_k)$, ${k \in \mathbb{N}}$, and $(X_\infty,d_\infty,\mathfrak{m}_\infty,p_\infty)$ are pointed metric measure spaces, we will say that $X_k$ converges to $X_\infty$ in the pointed measured Gromov-Hausdorff topology if there exist Borel measurable $\epsilon_k$--Gromov-Hausdorff approximations such that $\epsilon_k\rightarrow 0$ and $$(f_{\epsilon_k})_{\#}\mathfrak{m}_k\rightharpoonup \mathfrak{m}_\infty$$ weakly.

\end{defi}

Let $(X,d_X, \mathfrak{m})$ be a metric measure space, $x_0 \in X$ and $r>0;$ we consider the rescaled and normalized pointed metric measure space $(X, d_{rX}, \mathfrak{m}_r, x_0)$ where $d_{rX}(x,y) := rd(x,y)$ for all $x,y \in X$ and $\mathfrak{m}_r$ is given by
\begin{equation}\label{eq:normalized measure}
\mathfrak{m}_r:= \left(\int_{\overline{B}_{x_0}(1/r,X) )}1-d_{rX}(\cdot, x_0)\,d\mathfrak{m} \right)^{-1} \,\mathfrak{m}. 
\end{equation}
 
 
\begin{defi}[Tangent cone]
Let $(X,d,\mathfrak{m})$ be a metric measure space and $x_0 \in X$. A pointed metric measure space $(Y,d_Y,\mathfrak{n}, y_0)$ is a called a tangent cone of $X$ at $x_0$ if there exists a sequence $\{r_i\}\subset \mathbb{R},$ $r_i \rightarrow \infty$ so that $(X, d_{r_iX}, \mathfrak{m}_{r_i}, x_0)$ converges to $(Y,d_Y,\mathfrak{n},y_0)$ in the pointed measured Gromov-Hausdorff topology. We denote the collection of all tangent cones of $X$ at $x_0$ by $\Tan(X,x_0)$.
\end{defi} 

It is important to notice that tangent cones, unlike their analogues in Riemannian manifolds or Alexandrov spaces, may depend on the sequence one considers; examples of this behaviour are studied in \cite{CheegCol}.

\subsection{Ricci curvature bounds on metric measure spaces}
\mbox{}

Next we define the distortion coefficients  $\sigma_{K,N}^{(t)}(\theta)$ for $K \in \mathbb{R}$ and $N \in [1, \infty):$ 
\begin{equation*}
\sigma_{K,N}^{(t)}(\theta):=
\begin{cases}
\infty & \text{if }\, K\theta^2\geq N\pi^2, \\
\frac{\sin(t\theta \sqrt{K/N})}{\sin(\theta\sqrt{K/N})} & \text{if }\, 0<K\theta^2<N\pi^2,\\
t & \text{if }\, K\theta^2=0,\\
\frac{\sinh(t\theta\sqrt{-K/N})}{\sinh(\theta\sqrt{-K/N})} & \text{if } K\theta^2<0.
\end{cases}
\end{equation*}

\begin{defi}[$CD^*(K,N)$ curvature condition] 
 Let $K \in \mathbb{R}$ and $N \in [1, \infty)$. 
 A metric measure space $(X,d, \mathfrak{m})$ is said to be a $CD^*(K,N)$ space if for any two measures $\mu_0,\mu_1 \in \mathbb{P}_2(X)$ with bounded support contained in $\supp(\mathfrak{m})$ and with $\mu_0,\mu_1 \ll \mathfrak{m},$ there exists an optimal coupling $\pi\in \OptGeo(\mu_{0},\mu_{1})$ such that for every $t \in [0,1]$ and $N' \geq N$ one has, 
 \[
 -\int \rho_t^{1-\frac{1}{N'}}\,d\mathfrak{m} \leq -\int \sigma_{K,N'}^{(1-t)}(d(\gamma_0,\gamma_1))\rho_0^{-\frac{1}{N'}}+ \sigma_{K,N'}^{(t)}(d(\gamma_0,\gamma_1))\rho_1^{-\frac{1}{N'}}\,d\pi(\gamma), 
\] 
 where $t\in [0,1]$, and $\rho_t$  is the Radon-Nikodym derivative $d(e_{t\#}\pi)/d\mathfrak{m}$.

\end{defi}

Let $\Lip(X)$ denote the set of Lipschitz functions on $X$. For every $f \in \Lip(X),$ the local Lipschitz constant at $x,$ $|Df|(x),$ is defined by 
$$|Df|(x):= \limsup_{y\rightarrow x} \frac{|f(x)-f(y)|}{d(x,y)}, $$when $x$ is not isolated; otherwise $|Df|(x):= \infty$.

For $f \in L^2(X, \mathfrak{m})$ we define the \emph{Cheeger energy} of $f$ as  
\[
 \Ch(f):= \frac{1}{2}\inf\left\lbrace \liminf_{n\rightarrow \infty}\int_X|Df_n|^2d\,\mathfrak{m}\,\bigg|\, f_n \in \Lip(X), f_n \rightarrow f \, \text{ in } \, L^2 \right\rbrace. 
\] 
 Set $D(\Ch):= \lbrace f \in L^2(X,\mathfrak{m}); \Ch(f)< \infty \rbrace$.

We say that $(X,d, \mathfrak{m})$ is \emph{infinitesimally Hilbertian} if the Cheeger energy is a quadratic form. It is equivalent to the Sobolev space $W^{1,2}(X,d, \mathfrak{m}):= \lbrace f \in L^2\cap D(\Ch) \rbrace$ equipped with the norm $\|f\|^2_{1,2}:= \|f\|^2_2+2\Ch(f)$ being a Hilbert space.

\begin{defi}[$RCD^*(K,N)$ curvature condition]
If  a  $CD^*(K,N)$--space is infinitesimally Hilbertian then  it is called an $RCD^*(K,N)$ space. 
\end{defi}

For a fixed $(K,N)$ with $N\geq 1$, the set of normalized $RCD^*(K,N)$--spaces is precompact for the pointed measured Gromov-Hausdorff topology. As a consequence, every possible tangent cone at a point of an $RCD^*(K,N)$--space lies in $RCD^*(0,N)$.

\subsection{Structure of $RCD^{*}(K,N)$--spaces}
\mbox{}

To finish this section we state some structure results about $RCD^*(K,N)$ spaces.

\begin{teo} [Hausdorff dimension \cite{SturmII}]
\label{thm:Sturm_bound_on_dimension}
Let $(X,d,\mathfrak{m})$ be a $CD^*(K,N)$ space, with $K\in \mathbb{R}$ and $N \geq 1$. Then the Hausdorff dimension of $X$ is bounded above by $N$.
\end{teo}


The next three results were obtained by Gigli-Rajala-Sturm in \cite{GigRajStu} and regard essential non-branching on $RCD^*(K,N)-$spaces.
The first  gives concerns optimal plans in $RCD^*(K,N)$-spaces.
\begin{teo}[\cite{GigRajStu}]
\label{non-branch}
For $K\in \mathbb{R},$ and $N \in [1, \infty)$, let $(X,d, \mathfrak{m})$ be an $RCD^*(K,N)-$space.
Then for every $\mu, \nu \in \mathbb{P}_2 (X)$ with $\mu \ll \mathfrak{m}$ there exists a unique plan 
$\pi \in \OptGeo (\mu, \nu).$ Furthermore, this plan is induced by a map and concentrated on a set of non-branching geodesics, i.e. there is a Borel set $\Gamma \subset C([0,1], X)$ such that $\pi (\Gamma)=1$ and for every $t \in [0,1)$ the map $e_t: \Gamma \rightarrow X$ is injective.
\end{teo}

The second yields absolute continuity of measures appearing in interior points of geodesics in the Wasserstein space.
\begin{cor}[\cite{GigRajStu}]
\label{abscont}
Let $K\in \mathbb{R},$ $N \in [1, \infty)$ and $(X,d, \mathfrak{m})$ be an $RCD^*(K,N)-$space. If
 $\mu, \nu \in \mathbb{P}_2 (X)$ with $\mu \ll \mathfrak{m}$ and  $\pi \in \OptGeo (\mu, \nu)$ is the optimal plan given by Theorem \ref{non-branch}, then $e_{t\#}\pi \ll \mathfrak{m}$ for all $t \in [0,1).$
\end{cor}

The third assures that any given point can be connected by at most one geodesic to almost every other point.
\begin{cor}[\cite{GigRajStu}]
\label{uniquegeod}
Let $K \in \mathbb{R},$ $N\in [1,\infty)$ and $(X,d,\mathfrak{m})$ an $RCD^*(K,N)$ space. Then for  
every $x \in supp(\mathfrak{m})$ the following holds: for $\mathfrak{m}-$a.e. $y$ there is only one 
geodesic connecting $y$ to $x.$ 
\end{cor}

Finally, we will also need the stratification developed by Mondino and Naber.
\begin{teo}[\cite{MonNab}]
\label{thm:stratification}
Let $(X,d,\mathfrak{m})$ be an $RCD^{*}(K,N)$-space with $1\leq N\leq \infty,$ $K \in \mathbb{R}$. Define for $1\leq k\leq {\lfloor N \rfloor} $, 
\[
\mathcal{R}_k := \left\lbrace\, x \in X \,|\, \mathbb{R}^k \in \Tan(X,x) \text{ but } \mathbb{R}^{k+j} \notin \Tan(X,x) \text{ for every } j\geq 1\,\right\rbrace.
\]
 Then

\begin{enumerate}
\item Every $\mathcal{R}_k$ is $\mathfrak{m}-$measurable and $\mathfrak{m}\left(X- \bigcup_{1\leq k \leq {\lfloor N \rfloor} }\mathcal{R}_k \right) =0$.

\item For $\mathfrak{m}-$a.e. $x \in \mathcal{R}_k$ the tangent cone of $X$ at $x$ is unique and isomorphic to the $k-$dimensional Euclidean space.

\end{enumerate} 

\end{teo}

This stratification can be rewritten as  $\mathcal{R}_k = \cap _{\epsilon >0} \cup_{\delta >0} (\mathcal{R}_k)_{\epsilon, \delta}, $ where 

$$(\mathcal{R}_k)_{\epsilon, \delta} := \lbrace\, x \in X\, |\, d_{GH}(B_x (1, r^{-1}X),\, B_0 (1, \mathbb{R}^k)  ) < \epsilon,  0<r<\delta\, \rbrace,$$
which is the form that we will use in the rest of the paper.

\section{The isometry group of a metric space}
We will denote by $\Iso(X)$ the isometry group of a metric space $X$.
The natural topology when working with $\Iso(X)$ is the compact-open topology; \cite{Hatcher} contains a good summary of it. For the particular case of isometry groups, the following result is specially useful:

\begin{thm}[\cite{DanWae}, \cite{Koba}]
\label{thm:general properties iso(X)}
Let $(X,d)$ be a connected, locally compact metric space and $\Iso(X)$ its isometry group. For each $p\in X$, denote by $\Iso(X)_p$ the isotropy group of $\Iso(X)$ at $p$. Then $\Iso(X)$ is locally compact with respect to the compact open topology, and $\Iso(X)_p$ is compact for every $p$. Moreover, if $X$ is compact, $\Iso(X)$ is compact. 
\end{thm}

We will assume from now on that the space $X$ satisfies the hypothesis of Theorem \ref{thm:general properties iso(X)}.

The action of $\Iso (X)$ on $X$ is effective,  so   $\lbrace \Id \rbrace = \cap_{p \in X}(\Iso (X))_p. $ Thus, we have that  $\lbrace \Id \rbrace$ is closed, and  then that $\Iso (X)$ is Hausdorff.
Standard results as in \cite[Theorem 11.8]{Foll} 
give us that  each isotropy group         
$\Iso (X)_p$ admits a Haar measure.

The isometry group of an arbitrary metric space is a topological group, but there are many examples where it lacks a differentiable structure, as the reader can see in \cite{FukYam}. 
In order for $\Iso(X)$ to be a Lie group, the metric space needs to have additional structure, as happens, for instance in Riemannian manifolds and Alexandrov spaces. 

%
%

The main tool to distinguish  Lie groups from generic topological groups is the following useful fact.
\begin{teo}[\bf Gleason-Yamabe]\label{GleasonYamabe}
Let $G$ be a locally compact topological group. Suppose that $G$ is not a Lie group. Then, for every neighbourhood $U$ of the identity element in $G$, there exists a nontrivial compact subgroup of $G$ contained in $U$.
\end{teo}

In the rest of the paper, and when $G$ is not a Lie group,  we will refer to the sequence of subgroups obtained by the above Theorem when taking a basis of neighbourhoods of $G$ at the identity as \emph{small subgroups}. 
\subsection*{Actions of the isometry group}

Recall that an action $G\times X\to X$ is called \emph{proper} if the map 
\[
G\times X \to X\times X, \qquad (g,x)\to (g\cdot x, x)
\]
is a proper map. 
\begin{thm}
If $X$ is a complete metric space, the action of $\Iso(X)$ on $X$ is proper.
\end{thm}
A proof of this can be found, for instance,  in \cite{AleBet}; although their proof is for Riemannian manifolds, all their arguments are metric and can be carried out to the context of metric spaces.

Recall that $\Geo(X)$ denotes the set of constant speed geodesics  from $[0,1]$ to $(X,d)$ equipped with the $\sup$ norm, and  $e_t: \Geo(X)\rightarrow X$ is the evaluation map at $t\in [0,1]$, defined as 
\[
e_t(\gamma)=\gamma_t.
\]
Since isometries carry geodesics to geodesics, there is a natural action 
\[
\Iso(X)\times \Geo(X)\to \Geo(X),
\qquad
(g\cdot\gamma)_t=g\cdot \gamma_t
\]
that is itself isometric. Moreover,
 for every $t \in [0,1]$, geodesic $\gamma$ and $g \in \Iso(X),$  we have that
\[
e_t\circ g (\gamma)= g\circ e_t(\gamma). 
\]

$\Iso(X)$ also acts on $\mathbb{P}_2(X)$ by isometries, and so if $(\mu_t)_{t \in [0,1]}$ is a Wasserstein geodesic, $g \in \Iso(X)$  then  $(g_{\#}\mu_t)_{t \in [0,1]}$ is also a Wasserstein geodesic.

\subsection*{Quotient spaces by the action of the isometry group}
We know  that $G := \Iso(X)$ acts properly on $(X,d,\mathfrak{m}),$ so  \cite[Theorem  4.3.4]{Palais} yields that

$$d^* (G\cdot x, G \cdot y ) := \inf \lbrace\, d(x,gy) \,|\, g \in G\, \rbrace $$
is a metric on $X/G.$

It is a well-known fact that 
the quotient map $\pi: X \rightarrow X/G$ is an open map as well as a submetry (i.e, sendings balls $B_p(R,X)$ to balls $B_{\pi(p)}(R,X/G)$ for small $R$'s depending on $p$, see \cite{BerGui}). 

\begin{prop}\label{prop.loccomp}
$(X/G, d^*)$ is locally compact.
\end{prop}

\begin{proof}
Let $G\cdot p \in X/G.$  For small $\epsilon >0,$  the ball $A:=B_{p}(\epsilon, X)$ is precompact, and $\pi (\overline{A})$ is a compact set.
Since $\pi$ is open, $\pi (B_{p}(\epsilon /2, X))$ is an open neighbourhood of $G\cdot p$. Then the closure of   $\pi (B_{p}(\epsilon /2, X))$ is a subset of the compact set $\pi(\overline{A})$, and therefore it is itself compact. 
\end{proof}

\begin{prop}\label{prop.complete}
$(X/G, d^*)$ is a complete metric space.
\end{prop}

\begin{proof}
\footnote{This proof appears as Theorem 1.100 in some upublished notes by Holopainen.}   
Let $\lbrace G\cdot x_i \rbrace_{i \in \mathbb{N}}$ be a Cauchy sequence in $X/G.$ Consider an increasing sequence of indices, $\lbrace n_j \rbrace_{j \in \mathbb{N}}$ ,   such that:

$$d^* (G\cdot x_i, G \cdot x_{n_j}   ) < \frac{1}{2^j}, \quad  \forall  i \geq n_j.  $$
We will define now a Cauchy sequence $\lbrace y_j \rbrace_{j \in \mathbb{N}}$ in $X.$ For $y_1$ pick any point in $G_{x_1}.$ The rest of points are defined recursively as:

$$y_j \in G\cdot x_{n_j} \cap B_{y_{j-1}}( \frac{1}{2^{j-1}}, X  ). $$ 
It is clear that the resulting sequence is Cauchy in $X$, and by completeness, there exists some $y \in X$ such that $y_j \rightarrow y$. Finally notice that 

$$d^* (G\cdot y_j, G\cdot y) \leq d(y_j, y), $$ so then the original sequence converges to $G\cdot y.$
\end{proof}

\begin{prop}\label{prop.length}
$(X/G, d^*)$ is a length space.
\end{prop}

\begin{proof}
\footnote{The proof is the same as that of Proposition 3.2 in some unpublished notes of U.Lang}
Take $G \cdot x \neq G\cdot y $ so then $d^* (G\cdot x, G\cdot y) >0.$ Recall that the length $\mathcal{L}(\sigma)$ of any curve $\sigma $ in $X/G,$ with endpoints $G\cdot x$ and $G\cdot y $,  is greater or equal than $d^* (G\cdot x, G\cdot y) .$      

 Let $\epsilon >0.$ Consider points $x_1,y_1, \cdots , x_k, y_k  \in X$ such that:
\begin{itemize}
\item  $x_1 \in G\cdot x,  y_k \in G\cdot y.$

\item  $G\cdot y_j = G\cdot x_{j+1} ,$    for $j = 1, \cdots k-1.$ 

\item  $\sum_{j=1}^{k} d(x_j,y_j)  < d^*(G\cdot x, G \cdot y) + \frac{\epsilon}{2}.$
\end{itemize}
Consider now curves $\sigma_j : [0,1] \rightarrow X$ joining $x_j$ to $y_j$ such that  $\mathcal{L }(\sigma_j)< d(x_j,y_j)+\frac{\epsilon}{2k}.$  Define $\bar{\sigma} : [0,k] \rightarrow X/G$ as the concatenation of the curves  $\pi \circ \sigma_1,\pi \circ \sigma_2, \cdots , \pi \circ \sigma_k. $ Since the map $\pi: X \rightarrow X/G$ is a submetry we have that:

$$\mathcal{L}(\bar{\sigma})= \sum_{j=1}^{k} \mathcal{L}(\pi \circ \sigma_j) \leq \sum_{j=1}^{k} \mathcal{L}( \sigma_j) <\sum_{j=1}^{k} d(x_j,y_j)+ \frac{\epsilon}{2}  < d^*(G\cdot x, G \cdot y) + \epsilon.  $$

So we can conclude that 
$$
d^* (G\cdot x, G\cdot y )= \inf \lbrace \mathcal{L}(\sigma)\, |\,  \sigma \text{ has endpoints } G\cdot x, G\cdot y \rbrace.
$$ 
 \end{proof}

Propositions \ref{prop.loccomp}, \ref{prop.complete}, and \ref{prop.length} let us conclude by  \cite[Theorem 2.5.23]{Bur}  that

\begin{teo}
\label{thm:quotient length space}
$(X/G, d^*)$  is a geodesic space.
\end{teo}

\begin{cor}[Existence of geodesic lifts]
\label{cor:geodesic lift}
Let $x, y\in X$, and suppose that 
$d^*(G\cdot x, G\cdot y)=\ell$. Then there are geodesics $\gamma^*:[0,\ell]\to X/G$ with $\gamma^*(0)=G\cdot x$, $\gamma^*(\ell)=G\cdot y$ and  $\gamma:[0,\ell]\to X$ with $\gamma(0)=x$ such that  $\pi\circ\gamma=\gamma^*$. 
\end{cor}
\begin{proof}
Recall that $\pi:X\to X/G$ is a submetry, so $\pi(B_p(\ell,X))=B_{\pi(p)}(\ell,X/G)$; choose some $y'\in B_p(\ell,X)$ in the fiber of $\pi$ over $G\cdot y$. Since $\pi$ does not increase distances, $d(x,y')=\ell$. Let $\gamma$ be a geodesic in $X$ between $x$ and $y$, and $\gamma^*=\pi\circ \gamma$. It is clear now that $\gamma$ and $\gamma^*$ satisfy the requirements of the Corollary. 
\end{proof}

\section{The isometry group of an $RCD^{*}(K, N)$--space.}
In this section, we show that the isometry groups of 
$RCD^{*}(K, N)$--spaces are Lie groups. As mentioned in the introduction, this result has also been proved by G. Sosa in \cite{Sosa}, but our proof differs from his.

Lemma 4.1 in \cite{Sosa} states the following Lemma for a more general class of metric spaces; our proof here is different from that in \cite{Sosa}.

\begin{lem}
\label{lem.fixedpoints}
Let $X$ be an $RCD^*(K,N)$--space, and 
let $g \in \Iso (X)$ with  $g \neq \Id$. Then $\Fix(g)$, the fixed point set of $g$,    has $\mathfrak{m}-$measure zero.
\end{lem}

\begin{proof}
The proof proceeds by contradiction.  Suppose that $\mathfrak{m}(\Fix(g))>0$ where $g \neq \Id$. 
Choose some point $\mathfrak{m}$--density point $x_0 \in \Fix(g)$ and some $R>0$ such that 
\[
0< \mathfrak{m} (\Fix(g)\cap \overline{B_{x_0} (R,X)} ) < \infty.
\]

 Define 
 \[
 F:=\Fix(g)\cap \overline{B_{x_0} (R,X)}
 \] 
 and the measure  
 \[
 \mu_0 := \frac{\mathfrak{m}\llcorner F}{\mathfrak{m}(F)} \ll \mathfrak{m}.
 \]
 
We will denote the isotropy group of a point $x$ by $G_x$. 
Since  $X-\Fix(g)$ is a non-empty open set, there exists a ball  $B_{x}(r', X) \subset X-Fix(g).$ Consider  $\mathbb{H}$ the Haar measure of $G_{x_0}$ and define the measure $\mu_1$ on $X$ as
\[
\mu_1  := \int_{G_{x_0}} h_{\#}\delta_x \, d\mathbb{H}(h). \]

Since $g\in G_{x_0}$, $g_{\#}\mu_1= \mu_1,$ and  $\mu_1 (X) =\mu_1(G_{x_0}\cdot x)=1. $ We also  have that, since  $d(x_0, G_{x_0}\cdot x)= d(x_0,x)$,  $\mu_1$ actually belongs to the set $\mathbb{P}_2(X)$.  

Observe that $G_{x_0}(x) \cong  G_{x_0}/G_{x_0}\cap G_x$,  so for the quotient map $p: G_{x_0} \rightarrow G_{x_0}/G_{x_0}\cap G_x $ and for $\epsilon >0,$  it follows that $p^{-1}(B_{x}(\epsilon, X)\cap G_{x_0}(x)) \subset G_{x_0} $ is open. Then

$$\mu_1 (B_{x}(\epsilon, X)) = \int_{G_{x_0} } h_{\#}\delta_x (B_{x}(\epsilon, X))\, d\mathbb{H}(h) = \mathbb{H}(p^{-1}(B_{x}(\epsilon, X)\cap G_{x_0}(x)     )  ) >0.  $$

So we have that $\mu_1 (X-Fix(g) ) >0$. Therefore Theorem \ref{non-branch} yields a unique optimal plan $\Pi \in \OptGeo (\mu_0,\mu_1)$ concentrated on a non-branching set of geodesics $\Gamma.$ Furthermore, the evaluation maps $e_t : \Gamma \rightarrow X$ are injective for all $t \in [0,1).$

The two measures $\mu_0$ and $\mu_1$ are left fixed by $g$, and by uniqueness of the plan $\Pi$, we must have that  $\Pi = g_{\#}\Pi$. But because $g_{\#}\Pi$ is concentrated on $g\Gamma,$  it must  follow that $\Gamma$ agrees (as a set) with $g\Gamma.$ 

Finally, since  $\mu_1 (X-Fix(g) ) >0 $,  there exist  geodesics $\gamma$, $g\gamma \in \Gamma$ such that  $\gamma_1 \neq g\gamma_1 $ (i.e. $\gamma  \neq g\gamma$)  and $\gamma_0 = g\gamma_0.$ This contradicts the fact that  $e_0 :\gamma \rightarrow X$ is injective.

\end{proof}

Lemma \ref{lem.fixedpoints} implies that if $\Gamma$ is a compact non-trivial subgroup of $\Iso(X)$, then $X-\Fix(\Gamma)$ is a set of full measure. Therefore, if $\Gamma_n$ is a countable collection of subgroups of $\Iso(X)$, there is a full measure set of points in $X$ that are not fixed by any of the $\Gamma_n$'s.

\begin{defi}
Given a compact subgroup $\Gamma \leq \Iso(X)$, $R > 0,$ and a point $p\in X$, we define the \emph{displacement function}
$$
\Disp (\Gamma, R, p) := \max\lbrace  \,d(gq,q)\, |\, g \in \Gamma, \, q \in \overline{B_{p}(R/2,X)}\, \rbrace.
$$ 
For $R=0$ we define the displacement in a compatible way as 
\[
\Disp (\Gamma,0, p) := \max \lbrace \,d(gp,p)\,|\, g \in \Gamma\, \rbrace. 
\]
\end{defi}

\begin{lem}
\label{lem:displacement}
Let $\lbrace \Gamma_n \rbrace_{n \in \mathbb{N}} \subset \Iso(X)$ be  a sequence of compact non-trivial small subgroups, and let 
$p\not\in \Fix(\Gamma_n)$ for all $n$.
Then for every $\delta>0$, there exists a positive integer $n$, and some $0<\lambda<\delta$ such that
\[
\Disp(\Gamma_n,\lambda,p)=\dfrac{\lambda}{20}.
\] 
\end{lem}
\begin{proof}
From the definition of small groups and of the compact-open topology, it is clear that for every $p \in X$,
\begin{equation}\label{limit.smallgroups}
\lim_{n \rightarrow \infty}  \Disp(\Gamma_n, \delta,p) =0.
\end{equation} 
Thus there exists $n \in \mathbb{N}$ such that 
$$0 < \Disp (\Gamma_n,0,p) \leq \Disp (\Gamma_n,\delta,p) \leq \frac{\delta}{20}.$$ 

On the other hand, for any $\theta$ with  $0 < \theta <  \Disp (\Gamma_n,0,p)$, we have
\[
0 < \theta < \Disp (\Gamma_n,0,p) \leq \Disp (\Gamma_n,\theta,p),
\] 
thus 
$$
\frac{\theta}{20}\leq \Disp (\Gamma_n,\theta,p).
$$

%
%
%
%
%
Remember that  for fixed $n\in \mathbb{N}$ and $p \in X,$ the map $s \mapsto \Disp (\Gamma_n,s,p)$
is continuous.  By the intermediate value theorem there exists $\theta < \lambda < \delta$ such that
$ \Disp (\Gamma_n,\lambda,p) = \frac{\lambda}{20}.$ 
\end{proof}

We prove now the first main result of this paper.

\liegroup*


\begin{proof}
We will proceed by contradiction.
Assume that $\Iso (X)$ is not a Lie group, and  let $\lbrace \Gamma_n \rbrace_{n \in \mathbb{N}} \subset \Iso(X)$ denote the sequence of compact non-trivial small subgroups provided by  Theorem \ref{GleasonYamabe}.
 
From Lemma \ref{lem.fixedpoints}, we have that  $\mathfrak{m}(\Fix(\Gamma_n))=0$ for every $n \in \mathbb{N}$, and then the set 
\[
\cap_{n \in \mathbb{N}} (X-\Fix(\Gamma_n))
\] 
is of full measure. 
From the definition of the strata 
$\mathcal{R}$, for any $p \in \cap _{n \in \mathbb{N}}( X-\Fix(\Gamma_n)) \cap \mathcal{R} $ and any $\epsilon >0,$  there exists some $\delta >0$  and $1 \leq k \leq \lfloor N \rfloor$ such that $p \in (\mathcal{R}_k)_{\epsilon, \delta}.$ 
Choose some such $p$ in what follows.

Now consider a sequence $\epsilon_n \rightarrow 0$. By Lemma \ref{lem:displacement}, we can find  sequences of subgroups $\lbrace \Gamma_n \rbrace$ and of positive numbers $\lbrace \lambda_n \rbrace$ approaching zero such that 

$$\Disp(\Gamma_n,\lambda_n,p) =\frac{\lambda_n}{20}, \quad  d_{GH}(B_p (1, \lambda_n^{-1}X),B_0 (1,\mathbb{R}^{k}))< \epsilon_n. $$ 

By  Theorem \ref{thm:fukaya-yamaguchi}, there exists a closed subgroup $\Gamma$ in the isometry group of $\mathbb{R}^k$ such that $(B_p (1, \lambda_n^{-1}X), \Gamma_n , p)$ Gromov-Hausdorff converges equivariantly  to $(B_0 (1,\mathbb{R}^{k}),\Gamma, 0).$  Consider the functions:

\begin{itemize}
\item $f_n : B_p (1, \lambda_n^{-1}X)  \rightarrow B_{0}(1,\mathbb{R}^k),$

\item $\varphi_n : \Gamma_n (\frac{1}{\epsilon_n}) \rightarrow \Gamma(\frac{1}{\epsilon_n}),$

\item $\psi_n: \Gamma(\frac{1}{\epsilon_n}) \rightarrow \Gamma_n (\frac{1}{\epsilon_n}),$
 \end{itemize}
appearing in the definition of equivariant Gromov-Hausdorff convergence.

For each $n\in \mathbb{N}$ there exist $x_n \in B_p (1/2, \lambda_n^{-1}X)$ and $g_n \in \Gamma_n$ such that 
\[
d(g_n x_n,x_n) = \Disp (\Gamma_n, \lambda_n, p) = \frac{\lambda_n}{20}.
\]
Then
\begin{align*}
\frac{1}{20}  = \lambda_n^{-1}d(g_n x_n,x_n) &< \epsilon_n + d_E (f_n(x_n), f_n (g_n x_n ) ) \\
																		 &< 2\epsilon_n +d_E (f_n(x_n), \varphi_n(g_n)f_n(x_n)),
\end{align*}
so the group $\Gamma$ is not trivial. 

Take $y \in B_{0}(1/2, \mathbb{R}^k ) $ and $h \in \Gamma$. For $n$ large enough $h \in \Gamma(\frac{1}{\epsilon_n}).$ By Remark 27.18 of Villani \cite{Vill}  there also exists $x \in B_p (1/2,\lambda_n^{-1}X ) $ such that $d_E ( f_n(x), y)  < 3\epsilon_n.$ So we have
\begin{align*}
d_E (y,hy) &\leq  d_E (y, f_n(x) )+ d_E (f_n (x), hy) \\
				&\leq 3\epsilon_n + d_E (h^{-1} f_n (x ), f_n(\psi_n(h^{-1})x ) )+d_E (f_n(\psi_n(h^{-1})x ) ,y) \\
				&\leq 8\epsilon_n + \lambda_n^{-1}d(\psi_n (h^{-1})x,x) \\
				&\leq 8\epsilon_n + \frac{1}{20}.
\end{align*}
Since $\epsilon_n\to 0$, it follows that $\Disp(\Gamma,1,0 ) \leq \frac{1}{20}$.
This contradicts well-known properties of the isometry group of the Euclidean space, that rule out the existence of nontrivial compact subgroups in $\Iso(\mathbb{R}^k)$ with such a displacement function.
\end{proof}

\section{Upper bounds on the dimension of the isometry group}

In this section, we will find upper bounds for the dimension of the isometry group of an $RCD^*(K,N)$-space. 
To facilitate the reading, we start by bounding the dimension of the isotropy group of a point; recall that this is defined as the subset of elements of $\Iso(X)$ fixing that point. 

\begin{prop}\label{prop:isotropy upper bound}
Let $(X,d,\mathfrak{m})$ be an $RCD^*(K,N)$-space with $N \geq 1$, $K \in \mathbb{R}$. For any $p_0 \in X$, the topological dimension of the isotropy group $G_{p_0}$ is bounded above by $\frac{\lfloor N \rfloor (\lfloor N \rfloor -1)}{2}$. 
\end{prop}

\begin{proof}
Denote by $n$ the Hausdorff dimension of $X$; from Theorem \ref{thm:Sturm_bound_on_dimension}, and the relation between the topological and Haudorff dimensions of $X$, we get the inequalities
\[
\dim_{T}X \leq \lfloor n \rfloor\leq n\leq N.
\]

From \cite{GigPas} and \cite{KellMon}, there exist   $p_0 \in X$ and $R >0$, such that the ball $B:= B_{p_0}(R,X)$  has $n$-dimensional Hausdorff measure  $0< \mathcal{H}_{X}^{n}B < \infty $.  We define the function 
\[
F_{1}: B \rightarrow \mathbb{R},  \qquad F_{1}(q):= d(p_0,q).
\]
 This function is $1$-Lipschitz, so \cite[Proposition 3.1.5]{AmbTil} on the Hausdorff dimension of level sets of Lipschitz maps gives a constant $C>0$ such that
$$
\int_{[0,R]} \mathcal{H}_{X}^{n-1}(F^{-1}_{1}(t))\,d\mathcal{H}^{1}(t)\leq  C\, \mathcal{H}_{X}^{n}(B) < \infty.
$$ 
This implies that for $\mathcal{H}^1$-a.e. $t \in \mathbb{R}$,  the measure $\mathcal{H}_{X}^{n-1}(F^{-1}_{1}(t)) < \infty$ and the Hausdorff dimension $\dim_{\mathcal{H}}F^{-1}_{1}(t) \leq n-1$.
We pick $p_1 \in B$ such that  $\dim_{\mathcal{H}}F^{-1}_{1}(d(p_0,p_1)) \leq n-1$; notice that the orbit $G_{p_0}\cdot p_1 \subset  F^{-1}_{1}(d(p_0, p_1))$, thus $\dim_T G_{p_0}\cdot p_1\leq \lfloor n\rfloor-1$.

Recall that for a compact topological group $H$, $H\cdot p\simeq H/H_p$, so using this with $H=G_{p_0}$, we obtain 
\begin{align*}
\dim_{T}G_{p_0} &= \dim_{T}G_{p_0}\cdot p_1 +\dim_{T}G_{p_0}\cap G_{p_1} \\
						&\leq  \lfloor n\rfloor-1+ \dim_{T}G_{p_0}\cap G_{p_1}.
\end{align*}
Next, we define recursively the following functions for $1 < \ell \leq \lfloor n \rfloor$,  

\[
F_{\ell} : B \rightarrow \mathbb{R}^{\ell}, \qquad F_{\ell}(q) := (F_{\ell-1}(q), d(p_{\ell-1},q)).
\]  

These functions are again Lipschitz, so as before, for $\mathcal{H}^{\ell}$--a.e. point $x \in \mathbb{R}^\ell$, we have that  $\mathcal{H}^{n-\ell}_{X}(F^{-1}_{\ell}(x)) < \infty$ and the $(n-\ell)$--Hausdorff dimension of $F^{-1}_{\ell}(x)\leq n-\ell$. The next point $p_\ell$ is chosen in one of such level sets. Observe that at each $\ell$, 
the orbit
\[
\left(G_{p_0}\cap \cdots \cap G_{p_{\ell-1}}\right)\cdot p_{\ell}  \subset 
F^{-1}_{\ell}(F_{\ell-1}(p_\ell), d(p_{\ell-1},p_\ell)),
\]
and since 
\[
\mathcal{H}^{n-\ell}_{X}(F^{-1}_{\ell}(F_{\ell-1}(p_\ell), d(p_{\ell-1},p_\ell))) < \infty,
\]
we get that
\[
\dim_T\left(G_{p_0}\cap \cdots \cap G_{p_{\ell-1}}\right)\cdot p_{\ell}\leq \lfloor n\rfloor-\ell.
\]

%

Notice that, since
\[
\left(G_{p_0}\cap \cdots \cap G_{p_{l-1}}\right)\cdot p_{l}\simeq \left(G_{p_0} \cap \cdots \cap G_{p_{l-1}}\right)/ \left(G_{p_0}\cap \cdots \cap G_{p_{l}}\right),
\]
we get
\begin{align*}
\dim_{T} \left(G_{p_0} \cap \cdots \cap G_{p_{l-1}}\right)  = &\dim_{T}\left( G_{p_0}\cap \cdots \cap G_{p_{l-1}}\right)\cdot p_{l}+\dim_{T} \left(G_{p_0}\cap \cdots \cap G_{p_{l}}\right)
\\
&\leq \lfloor n\rfloor -\ell + \dim_{T} \left(G_{p_0}\cap \cdots \cap G_{p_{l}}\right).
\end{align*}

So, after iterating the above procedure enough times, we reach

$$
\dim_{T}G_{p_0} \leq \frac{\lfloor n \rfloor ( \lfloor n \rfloor -1)}{2} + \dim_{T} \left(G_{p_0}\cap \cdots \cap G_{\lfloor n \rfloor -1}\right).
$$

Observe that $0 \leq n-\lfloor n \rfloor < 1,$ and that for  $\mathcal{H}^{\lfloor n \rfloor}$-a.e. point $x \in \mathbb{R}^{\lfloor n \rfloor}, $ $\mathcal{H}^{n-\lfloor n \rfloor}(F^{-1}_{\lfloor n \rfloor}(x)) < \infty $; this implies that $\dim_{T} F^{-1}_{\lfloor n \rfloor}(x) =0$. 
Then for all  points in the set  
\[ 
D:= \lbrace \, q \in B\, |\, \mathcal{H}^{n-\lfloor n \rfloor} (\,F^{-1}_{\lfloor n \rfloor}\,(F_{\lfloor n \rfloor -1}(q), d(p_{\lfloor n \rfloor -1}, q\,) \,) \,) < \infty \,\rbrace,
\] we have  

\begin{align*}
  \dim_{T} G_{p_0}\cap \cdots \cap G_{p_{\lfloor n \rfloor -1}} &= \dim_{T} G_{p_0}\cap \cdots \cap G_{p_{\lfloor n \rfloor -1}} (q) +\dim_{T}G_{p_0}\cap \cdots \cap G_{p_{\lfloor n \rfloor -1}}\cap G_q\\
  &= \dim_{T}G_{p_0}\cap \cdots \cap G_{p_{\lfloor n \rfloor -1}}\cap G_q.
\end{align*}   
We consider now the  connected components of the identity in these groups, that we denote $(G_{p_0}\cap \cdots \cap G_{p_{\lfloor n \rfloor -1}} )_{0}$ and $(G_{p_0}\cap \cdots \cap G_{p_{\lfloor n \rfloor -1}}\cap G_q)_{0}.$ Since they have the same topological dimension, they must be equal, so  
\[
\left(G_{p_0}\cap \cdots \cap G_{p_{\lfloor n \rfloor -1}}\right)_{0}\cdot q =q,
\]  
and all of $D$ is fixed by   
\[
(G_{p_0}\cap \cdots \cap G_{p_{\lfloor n \rfloor -1}})_{0}.
\]

 Now we prove that $D$ is dense in $B$: let $q \in B $ and $\epsilon >0$, and notice that for every $q' \in B_{q}(\epsilon, X)$  we have 
\[
\left|\, d(p_{l}, q) -d(p_{l},q' )\, \right| < \epsilon, \quad \text{for all } 1\leq \ell \leq \lfloor n \rfloor -1,
\] 
so we can easily find a $q' \in D\cap B_{q}(\epsilon, X) $. 

With this it is easy to see that the whole ball $B$ is fixed by $(G_{p_0}\cap \cdots \cap G_{p_{\lfloor n \rfloor -1}})_{0}$: for any $q \in B$, choose a sequence  $(q_{i})_{i=1}^{\infty} \subset D$ such that $q_{i } \rightarrow q$; then for $h \in (G_{p_0}\cap \cdots \cap G_{p_{\lfloor n \rfloor -1}})_{0}$, $q_{i} \rightarrow q$ implies that $hq_{i} \rightarrow q$.

Finally suppose that $(G_{p_0}\cap \cdots \cap G_{p_{\lfloor n \rfloor -1}})_{0}$ was not trivial. Then for a nonidentity element $h \in (G_{p_0}\cap \cdots \cap G_{p_{\lfloor n \rfloor -1}})_{0}$, there would exists $x \in X-B$ that is not fixed by $h$ and such that there exists a unique non-branching geodesic  $\gamma$ from $p_0$ to $x.$  Then $h\gamma$ is a geodesic from $p_0$ to $hx$, and for all $t \in [0,1)$ such that $\gamma_t \in B$, we have that $h\gamma_t\in B$ and  $\gamma_t = h\gamma_t$; this means that $\gamma$ must branch and we have a contradiction. 

We conclude that $(G_{p_0}\cap \cdots \cap G_{p_{\lfloor n \rfloor -1}})_{0}$ must be trivial and then 

$$
\dim_{T}G_{p_0}
\leq \frac{\lfloor n \rfloor( \lfloor n\rfloor -1) }{2}
 \leq \frac{\lfloor N \rfloor( \lfloor N \rfloor -1) }{2}. 
 $$

\end{proof}

\begin{rem}\label{rem:isotropy equality case}
In Section \ref{section:large isometry groups}, we will need to return to the proof of the above Proposition and consider the case when the equality is reached, i.e, 
\[
\dim_T G_{p_0} = \frac{\lfloor N \rfloor( \lfloor N \rfloor -1) }{2}.
\]
If this happens, then the above proof gives $\lfloor n\rfloor =\lfloor N\rfloor$.
\end{rem}

We are now ready to prove the second main result of this paper.

\dimensionbound*


\begin{proof}[Proof of Theorem \ref{thm:dimension}] Choose some $p_0\in X$, and consider  the \emph{Myers-Steenrod map} 
\[
F: G \rightarrow X, \quad  F(g) := gp_0.
\]
 It  is easy to see that it is continuous.  The quotient map $\pi: G \rightarrow G/G_{p_0}$  is a fibration, thus  we can find a compact subset $V$ of $G/G_{p_0}$ with nonempty interior, and a  continuous section $s: V \rightarrow G$ of $\pi$ passing through the identity.  The map  $\tilde{F} = F \circ s$ is injective and since $V$ is compact, it is a topological embedding. Then  the dimension of $V$, that agrees with that of $G/G_{p_0}$, is bounded above by ${\lfloor N \rfloor} $.

With this we conclude that 
\[
\dim G = \dim G/G_{p_0} + \dim G_{p_0} \leq \frac{{\lfloor N \rfloor} ({\lfloor N \rfloor} +1)}{2}
\]  
as desired.
\end{proof}

The Myers--Steenrod map was also used in \cite{GalGui} to obtain Theorem \ref{thm:dimension} for the case of Alexandrov spaces. 

\begin{rem}
\label{rem:compact_set_embedding}
Observe that the  above proof gives that there are compact sets $V$ in $G/G_{p_0}$ with nonempty interior that embed topologically in $X$.
\end{rem}

\begin{rem}
It is interesting to notice that the arguments in this section apply to measured metric spaces more general than $RCD^*(K,N)$-spaces; in fact, we only need for the space to have finite Hausdorff dimension, and such that any nontrivial isometry does not leave fixed an open set.  
\end{rem}

\begin{rem}
We could give an alternative proof to Theorem \ref{thm:dimension} using slices for the action of $\Iso(X)$ on $X$ that was pointed to us by A. Lytchak. The outline of the proof, with many details missing, is as follows: Since $G:=\Iso(X)$ is a Lie group and the action of $\Iso(X)$ on $X$ is proper, $X$ is a Cartan space in the terminology of \cite{Palais}. Proposition  2.3.1 in \cite{Palais} assures the existence of local slices for the action of $G$  through any point $p$ of $X$; moreover, by Proposition 2.1.4 in the same reference, the slice can be taken invariant for the action of the isotropy group of $p$.

Since $G$ is a Lie group, standard arguments give the existence of principal orbits for the action of $G$ on $X$. Choose some $p$ in such a principal orbit, and a 
slice $S$ through $p$ such that $G\cdot S$ is an open neighbourhood $U$ of $G/G_p$; once again, standard arguments show that for any point $x\in S$, $G_x=G_p$. 

Consider an arbitrary element $g\in G$ and its restricted action on $G\cdot p$; if it were trivial, then $g$ would fix the open set $U$, which is impossible because of Lemma \ref{lem.fixedpoints}. Thus, the action of $G$ on $G\cdot p$ is faithful and the required bound dimension follows easily.

\end{rem}

\section{$RCD^*(K,N)$-spaces with large isometry groups}
\label{section:large isometry groups}

In this section, we study what $RCD^*(K,N)$-spaces have isometry group of maximal dimension. We structure the proof in a series of Lemmas.

\begin{lem}
\label{lem:N_and_topological_dimension}
The topological dimension of any $RCD^*(K,N)$-space is less or equal than $\lfloor N\rfloor$.
\end{lem}
\begin{proof}
Theorem \ref{thm:Sturm_bound_on_dimension} says that the Hausdorff dimension of $X$ does not exceed $N$; on the other hand, Szpilrajn's inequality gives that the Hausdorff dimension of $X$ is greater or equal than the inductive dimension of $X$ (see \cite[Chapter 7]{HuWall}). But for Polish spaces, the topological and the inductive dimension coincide (\cite[Ch.4, corollary 5.10]{Pears}).
\end{proof}

\begin{lem}
\label{lem:maxdimension and dimension orbits}
Let $(X, d_{X}, \mathfrak{m})$  be an $RCD^*(K,N)$-space
. If the dimension of the isometry group of  $X$ is ${\lfloor N \rfloor} ({\lfloor N \rfloor} +1)/2,$ then 
\begin{enumerate}
\item $X$ has topological dimension equal to  ${\lfloor N \rfloor} $;
\item each orbit $G\cdot p_0$ has topological dimension equal to  ${\lfloor N \rfloor} $;
\item $\dim G_{p_0}=\frac{\lfloor N \rfloor( \lfloor N \rfloor -1) }{2}$
\end{enumerate}
\end{lem}
\begin{proof}

For the Lie group action of $G= \Iso(X)$ on $X$, we have that 
\[
\dim G/G_p + \dim G_p =\dim G.
\]
 Since there are compact neighbourhoods in the orbit $G/G_p$ that embed topologically in $X$ (see Remark \ref{rem:compact_set_embedding}), we have that $\dim G/G_p\leq \dim_T X$, and then
\begin{equation}\label{eq:rigidity_sum}
\dfrac{{\lfloor N \rfloor} ({\lfloor N \rfloor} +1)}{2}=\dim G\leq 
\dim_T X + \dim G_p\leq {\lfloor N \rfloor} +\dim G_p.
\end{equation}
From Proposition \ref{prop:isotropy upper bound} and equation  \eqref{eq:rigidity_sum}, we deduce that
\[
\dim G_{p_0}=\dfrac{{\lfloor N \rfloor} ({\lfloor N \rfloor} -1)}{2}.
\]

Once we have this, the first two claims follow from Remark \ref{rem:isotropy equality case}.

%
%
%
\end{proof}

\begin{lem}
\label{lem:transitive}
If $X$ is an $RCD^*(K,N)$-space with isometry group $G$ of maximal dimension, then $G$ acts transitively on $X$.  
\end{lem}

\begin{proof} 
We will proceed by contradiction. Suppose there are two different orbits $G\cdot p_0$ and $G\cdot p_1$. From 
Corollary \ref{cor:geodesic lift}
, we know that there is a geodesic $\gamma^*:[0,\ell]\to X/G$ connecting $G\cdot p_0$ to $G\cdot p_1$ that lifts to a geodesic
$\gamma:[0,\ell]\to X$ with $\gamma(0)=p_0$ and $\gamma(\ell)\in G\cdot p_1$. Without loss of generality, we can assume that $\gamma(\ell)=p_1$. 

Consider the map
\[
\Phi:G\times [0,\ell]\to X, \qquad \Phi(g,t)=g\cdot \gamma(t).
\] 
Observe that for fixed $t$, the image $\Phi(G,t)$ is the orbit of $G$ through $\gamma(t)$, and coincides with the set of points in the image of $\Phi$ at distance $t$ from $G\cdot p_0$. 

Choose the function $f:X\to\mathbb{R}$, $f(p)=d(p,G\cdot p_0)$, and restrict it to $\Img\Phi$; its level sets are clearly the orbits $G\cdot\gamma(t)$.  Since $f$ is Lipschitz, we can use again \cite[Proposition 3.1.5]{AmbTil} to conclude that almost every orbit $G\cdot \gamma(t)$ has Hausdorff dimension less or equal to $n-1$, where $n$ is the Hausdorff dimension of $X$ (technically, we should restrict $f$ to a compact subset of $\Img \Phi$ with non empty interior; this can be easily reached using a section of the fibration $G\to G/G_{p_0}$, but we will skip details to facilitate the reading).   

Using Lemma \ref{lem:maxdimension and dimension orbits}, we get  the chain of inequalities
\[
\dim_{\mathcal{H}}G\cdot\gamma(t)\leq n-1\leq N-1 <\lfloor N\rfloor=\dim_T G\cdot\gamma(t)
\]
which gives a contradiction.
\end{proof}

We are finally ready to prove the last  main result of the paper:

\maximaldimension*
\begin{proof}
Choose some $p_0\in\mathcal{R}_m$ where $m \in  [1, N] \cap \mathbb{N}.$
Lemma \ref{lem:transitive}  says that $G$ acts transitively on $X$, therefore $X$ is homeomorphic to $G/G_{p_0}$ and locally contractible.  Berestovskii \cite[Theorem 3]{Ber2} characterizes $X$ as isometric to the quotient of $G$ by $G_{p_0}$ endowed with a Carnot-Carath\'eodory-Finsler metric $d_{cc}$ that corresponds to a $k-$dimensional tangent distribution $\Delta$, and a norm $\mathcal{F}$ defined over it. In order to ensure that the metric is actually Riemannian, we now look at the tangent spaces corresponding to these structures. 

On the one hand, the tangent spaces of $X$  at $p_{0}$ are of the form $(\mathbb{R}^{m },d_{E},\mathcal{L}^m,0)$, obtained by means of measured Gromov-Hausdorff convergence; note that the Hausdorff dimension of this space is ${m} $.

On the other hand,  the tangent spaces of spaces of the form $(G/G_{p_0},d_{cc})$ are Carnot groups whose Hausdorff dimension is, by \cite{Montg} , greater or equal than ${\lfloor N \rfloor} $, with equality being achieved if and only if the dimension of $\Delta$ is exactly ${\lfloor N \rfloor} $. Note that this tangent space is obtained by using Gromov-Hausdorff convergence without any requirement on a measure.

By uniqueness of the tangent spaces we have that these must be isometric, so they must have the same Hausdorff dimension. It follows then that the distribution $\Delta$ has dimension ${\lfloor N \rfloor} $. From this we have by Theorem $7$ of \cite{Ber2} that $X$ is actually a Finsler manifold; since it is also an $RCD^{*}(K,N)$-space, it must have quadratic Cheeger energy, and therefore be in fact a Riemannian manifold \cite{AmbGigSav}. The Theorem follows now from Theorem $3.1$ of \cite{Koba}, that characterizes Riemannian manifolds with maximal isometry groups.

Finally note that by \cite{CheegCol2}  the measure $\mathfrak{n}$ that we give to the resulting Riemannian manifold and the ${\lfloor N \rfloor} -$dimensional  Hausdorff measure are mutually absolutely continuous. 
\end{proof}

\end{document}